\documentclass[11pt]{article}

\usepackage[a4paper,margin=1in]{geometry}
\usepackage{amsmath,amssymb,amsthm,mathtools}
\usepackage{enumitem}
\usepackage{microtype}
\usepackage{cite}
\usepackage[colorlinks=true,linkcolor=blue,citecolor=blue,urlcolor=blue]{hyperref}

\newtheorem{theorem}{Theorem}[section]
\newtheorem{lemma}[theorem]{Lemma}

\newtheorem{claim}[theorem]{Claim}

\theoremstyle{definition}
\newtheorem{definition}[theorem]{Definition}
\theoremstyle{remark}

\newcommand{\eps}{\varepsilon}
\newcommand{\calC}{\mathcal C}

\newcommand{\calL}{\mathcal L}
\newcommand{\calF}{\mathcal F}

\title{On the codegree threshold for Hamilton $\ell$‑cycles in $k$‑uniform hypergraphs\thanks{Supported by  National Key Research and Development Program of China 2023YFA1010203}}
\author{
  Hongliang Lu\thanks{
    School of Mathematics and Statistics,
    Xi'an Jiaotong University.
    Email: \texttt{luhongliang215@sina.com}.
  }
  \and
  Feihong Yuan\thanks{
    School of Mathematics and Statistics,
    Xi'an Jiaotong University.
    Email: \texttt{fhyuan1@gmail.com}.
  }
}
\date{}
\begin{document}

\maketitle

\begin{abstract}
        In this note, we resolve the remaining open case of a conjecture by Han and Zhao concerning the codegree threshold for Hamilton $\ell$-cycles in $k$-uniform hypergraphs. Specifically, we prove that for integers $k\ge 3$, $3k/4\le \ell<k$, with $k\not\equiv 0 \pmod{k-\ell}$, and for all sufficiently large $n$ divisible by $k-\ell$, every $n$-vertex $k$-uniform hypergraph $H$ satisfying  
\[
\delta_{k-1}(H)\ge \frac{n}{(k-\ell)\left\lceil \frac{k}{k-\ell}\right\rceil}
\]  
contains a Hamilton $\ell$-cycle. Our proof builds on the framework of Gan, Han and Xu, and refines their argument to obtain, at the exact threshold, the required family of paths.
\end{abstract}

\section{Introduction}
The Hamilton cycle problem is one of the central problems in graph theory. The celebrated theorem of Dirac~\cite{Dirac1952} asserts that every $n$-vertex graph with minimum degree at least $n/2$ contains a Hamilton cycle.  Over the last several decades, this classical theorem has been extended to the setting of uniform hypergraphs. 

A \emph{\(k\)-uniform hypergraph,} (or \emph{\(k\)-graph}), is a pair \(H=(V,E)\) where \(E\subseteq \binom{V}{k}\). For \(U\subseteq V(H)\), let \(H[U]\) denote the induced subhypergraph on \(U\), whose edge set is \( \left\{ e\in \binom{U}{k} : e\in E(H) \right\}.\) For a \((k-1)\)-subset \(S\subseteq V\), the \emph{degree} of \(S\) in \(H\) is defined as  
\[
\deg_H(S)=\bigl|\{v\in V\setminus S : S\cup\{v\}\in E(H)\}\bigr|.
\]  
The \emph{minimum codegree} of \(H\), denoted by \(\delta_{k-1}(H)\), is the minimum of \(\deg_H(S)\) over all \((k-1)\)-subsets \(S\subseteq V\).

Let \(1\le \ell<k\). An \emph{\(\ell\)-path} is a \(k\)-graph whose vertices admit a linear ordering such that each edge consists of \(k\) consecutive vertices, and any two consecutive edges intersect in exactly \(\ell\) vertices. In such an ordering, the ordered \(\ell\)-sets formed by the first and last \(\ell\) vertices in reverse order are called the \emph{ends} of the \(\ell\)-path. An {\emph{\(\ell\)-cycle} is defined analogously using a cyclic ordering, where again each edge consists of \(k\) consecutive vertices and consecutive edges share precisely \(\ell\) vertices. A \emph{Hamilton \(\ell\)-cycle} is an \(\ell\)-cycle containing every vertex of the hypergraph. In the case \(\ell=k-1\), an \(\ell\)-path (resp. \(\ell\)-cycle) is referred to as a \emph{tight path} (resp. \emph{tight cycle}).

For tight Hamilton cycles, Katona and Kierstead~\cite{KatonaKierstead1999} proved that a minimum codegree of at least $(1-\frac{1}{2k})n+4-k-\frac{5}{2k}$ is sufficient, and they conjectured the optimal bound to be $(n-k+3)/2$.  R\"odl, Ruci\'nski, and Szemer\'edi~\cite{RodlRucinskiSzemeredi2006,RodlRucinskiSzemeredi2008,RodlRucinskiSzemeredi2011} established the asymptotic threshold $(1/2+o(1))n$ for tight Hamilton cycles, and for $k=3$ they obtained the exact value $\lfloor n/2\rfloor$.  Markstr\"om and Ruci\'nski~\cite{MarkstromRucinski2011} showed that the same asymptotic threshold $n/2$ is tight for every $\ell$ satisfying $k= 0 \pmod{k-\ell}$.  When $k\not\equiv 0 \pmod{k-\ell}$, the threshold is smaller: H\`an and Schacht~\cite{HanSchacht2010} proved the asymptotic bound $(1/(2(k-\ell))+o(1))n$ for $1\le \ell<k/2$, and K\"uhn, Mycroft, and Osthus~\cite{KuhnMycroftOsthus2010} established the general asymptotic threshold $n/(\left\lceil k/(k-\ell)\right\rceil (k-\ell))$.

R\"odl and Ruci\'nski~\cite{RodlRucinski2010} raised the question of determining the exact minimum codegree threshold that forces a Hamilton $\ell$-cycle in $k$-graphs.  Several partial results on this problem have since been obtained. For the tight 3-uniform case \((k,\ell)=(3,2)\), R\"odl, Ruciński, and Szemerédi~\cite{RodlRucinskiSzemeredi2011} established the exact threshold.  The loose 3-uniform case \((k,\ell)=(3,1)\) was settled by Czygrinow and Molla~\cite{CzygrinowMolla2014} who showed that the threshold is \(\lceil n/4\rceil\). 
 Han and Zhao~\cite{HZ}  resolved the full range \(k\ge 3\) and \(1\le \ell<k/2\), proving the sharp bound \(\left\lceil n/(2(k-\ell))\right\rceil\). 
The boundary case \(\ell=k/2\) falls into the divisibility regime: Garbe and Mycroft~\cite{GarbeMycroft2018} determined the exact threshold for \((k,\ell)=(4,2)\), and Hàn, Han, and Zhao~\cite{HanHanZhao2022} determined the exact threshold for all even $k\ge6$.  For further results on Hamilton cycles and other spanning structures in hypergraphs, see~\cite{AlonSpencer2000,BastosMotaSchachtSchnitzerSchulenburg2017Approx,BastosMotaSchachtSchnitzerSchulenburg2018Sharp,BermondGermaHeydemannSotteau1978,BussHanSchacht2013,ChengHanWangWangYang2025,FurediJiangKostochkaMubayiVerstraete2020,GaoHanZhao2019,GlebovPersonWeps2012,HanSunWang2025,HanZhao2015Vertex,HanZhao2016,KeevashKuhnMycroftOsthus2011,KuhnOsthus2006,KuhnOsthus2014,LangSanhuezaMatamala2022,LuSzekely2007,ReiherRodlRucinskiSchachtSzemeredi2019,Szemeredi1978} and the references therein.

Han and Zhao~\cite{HZ} further proposed a conjecture addressing the non-divisibility case. 
Specifically, for integers \(k\ge 3\) and \(k/2<\ell<k\) with \(k\not\equiv 0 \pmod{k-\ell}\), the conjecture asserts that every sufficiently large \(n\)-vertex \(k\)-graph, where \(n\) is divisible by \(k-\ell\), satisfies the following: if  
\[\delta_{k-1}(H)\ge \frac{n}{\left\lceil \frac{k}{k-\ell} \right\rceil (k-\ell)},
\]  
 then \(H\) contains a Hamilton \(\ell\)-cycle. 
Recently, Gan, Han, and Xu~\cite{GHX} made substantial progress towards this conjecture. They proved the exact threshold for the range \(k/2<\ell<3k/4\), and further showed that for all \(k/2<\ell<k\) with \(k\not\equiv 0 \pmod{k-\ell}\), the same bound, up to an additive term of \(k^2/2\), is sufficient. Consequently, the only remaining case for which the exact threshold is yet to be determined is the range \(3k/4\le \ell<k\) under the non-divisibility condition \(k\not\equiv 0 \pmod{k-\ell}\).  
Our main result resolves this remaining case.

\begin{theorem}\label{thm:main}
Let $k$ and $\ell$ be positive integers satisfying 
\[
k\ge 3,\qquad \frac{3k}{4}\le \ell<k,\qquad k\not\equiv 0 \pmod{k-\ell}.
\]
Then there exists an integer $n_0=n_0(k,\ell)$ such that for every integer $n\ge n_0$ with $n\equiv 0 \pmod{k-\ell}$, every $k$-uniform hypergraph $H$ on $n$ vertices with 
\[
\delta_{k-1}(H)\ge \frac{n}{(k-\ell)\left\lceil \frac{k}{k-\ell}\right\rceil}
\]
contains a Hamilton $\ell$-cycle.
\end{theorem}
The bound in Theorem~\ref{thm:main} is sharp. Indeed, let $|A|=\left\lceil\frac{n}{(k-\ell)\lceil k/(k-\ell)\rceil}\right\rceil-1 $, and let \(H\) consist of all \(k\)-sets intersecting \(A\). Then \(\delta_{k-1}(H)=|A|\), while $H$ contains no Hamilton \(\ell\)-cycle.

Our proof employs the extremal/non-extremal framework developed by Gan, Han and Xu~\cite{GHX}. The main novelty of this paper lies in a refined path-supply argument tailored to the extremal case, which establishes that the sharp codegree condition forces the existence of a family of sufficiently short, vertex-disjoint \(\ell\)-paths whose ends are linkable. Consequently, the additive term of \(k^2/2\) present in the earlier result of~\cite{GHX} is eliminated.

\section{Preliminaries}\label{sec2}
In this section, we recall the definitions and results from~\cite{GHX,HZ} that will be used in the proof. Throughout the rest of the paper, let $s=\left\lceil k/(k-\ell)\right\rceil$. Since  $k\not\equiv 0 \pmod{k-\ell}$, there is a unique integer $r\in\{1,\ldots,k-\ell-1\}$ such that $k=(s-1)(k-\ell)+r$.  For two positive quantities $\alpha$ and $\beta$, the notation $\alpha \ll \beta$ means that $\alpha$ is chosen to be a sufficiently small positive constant depending only on $\beta$. % all statements involving such constants are understood to hold provided that $\alpha$ is chosen accordingly.

\begin{definition}Let $\Delta>0$.
An \(n\)-vertex \(k\)-graph \(H\) is called \(\Delta\)-extremal if there exists a set \(B\subseteq V(H)\) such that $|B|=\left\lfloor\left(1-\frac1{s(k-\ell)}\right)n\right\rfloor$ and $e(H[B])\le \Delta n^k$.
\end{definition}

The non-extremal case is fully handled by the following result of Gan, Han, and Xu~\cite{GHX}.
\begin{lemma}[\cite{GHX}]
\label{lem:ghx-nonext}
For every fixed $k\ge3$, every $1\le\ell<k$ with $k\not\equiv 0 \pmod{k-\ell}$, and every $0<\Delta<1$, there exists $\gamma>0$ such that the following holds for all sufficiently large $n$ with $n\equiv 0 \pmod{k-\ell}$. If $H$ is an $n$-vertex $k$-graph that is not $\Delta$-extremal and satisfies
\[
\delta_{k-1}(H)\ge \left(\frac{1}{s(k-\ell)}-\gamma\right)n,
\]
then $H$ contains a Hamilton $\ell$-cycle.
\end{lemma}

We shall focus on the extremal case. Choose a partition $V(H)=A\cup B$ such that
$|B|=\left\lfloor\left(1-\frac{1}{s(k-\ell)}\right)n\right\rfloor$, $|A|=\left\lceil\frac{n}{s(k-\ell)}\right\rceil$,
and, subject to this condition, $e(H[B])$ is minimum. For sufficiently small $\Delta$ and sufficiently large $n$, we may assume $e(H[B])\le \varepsilon_0\binom{|B|}{k}$ with $0<\varepsilon_0\ll 1$. Set $\varepsilon_1=\varepsilon_0^{1/4}$ and $\varepsilon_2=2\varepsilon_1^2$.

For \(X\subseteq V(H)\) and \(T\subseteq V(H)\) with \(|T|<k\), let
$ \deg_H(T,X):=
        \left|\left\{Y\in \binom{X}{k-|T|}: T\cup Y\in E(H)\right\}\right|$. If $T=\{v\}$, we simply write $\deg_H(v,X)$ for $\deg_H(\{v\},X)$.
Define
\[
        A'=\left\{v\in V(H):
        \deg_H(v,B)\ge (1-\varepsilon_1)\binom{|B|}{k-1}\right\},
\]
\[
        B'=\left\{v\in V(H):
        \deg_H(v,B)\le \varepsilon_1\binom{|B|}{k-1}\right\},
\]
and $V_0=V(H)\setminus(A'\cup B')$.

For \(0\le t\le \ell\), a \(t\)-set \(T\subseteq V(H)\) is called
\(\varepsilon_1\)-bad if $\deg_H(T,B)>\varepsilon_1\binom{|B|}{k-t}$;
otherwise \(T\) is called \(\varepsilon_1\)-good.  An ordered \(\ell\)-set
\(L=(v_1,\ldots,v_\ell)\) is called linkable if, for every
\(i\in[s-1]\), $\{v_{(i-1)(k-\ell)+1},v_{(i-1)(k-\ell)+2},\ldots,v_\ell\}$
is \(\varepsilon_1\)-good. The next lemma reduces the extremal case to finding a suitable collection of short paths inside \(B'\).

\begin{lemma}[\cite{GHX}]
\label{lem:ghx-completion}
Let integers \(k>\ell\ge2\) satisfy \(k\not\equiv 0\pmod{k-\ell}\), and let \(n\) be sufficiently large with \(n\equiv 0\pmod{k-\ell}\). Suppose \(H\) is a \(\Delta\)-extremal \(n\)-vertex \(k\)-graph, and let \(A\) and \(B'\) be defined as above. Put \(q=|A\cap B'|\). If \(H\) contains \(sq\) vertex-disjoint \(\ell\)-paths in \(B'\), each of length at most \(s\) and with two linkable \(\ell\)-ends, then \(H\) has a Hamilton \(\ell\)-cycle.
\end{lemma}
The main new ingredient of this paper is the following result.
\begin{theorem}
        \label{thm:paths}
       Suppose \(k\ge3\), \(k/2<\ell<k\), and \(k\not\equiv 0 \pmod{k-\ell}\). Let \(H\) be an \(n\)-vertex \(k\)-graph, where \(n\equiv 0 \pmod{k-\ell}\) is sufficiently large. Assume that \(H\) is \(\Delta\)-extremal for some \(0<\Delta\ll1\), and that \(A',B',V_0\) are defined as above. If \(\delta_{k-1}(H)\ge \frac{n}{s(k-\ell)}\) and \(|A\cap B'|=q>0\), then \(H\) contains \(sq\) vertex-disjoint \(\ell\)-paths in \(B'\), each of length at most \(s\) and each with two linkable ends.
\end{theorem}
The following facts, taken from \cite[Claims~3.1 and~3.2]{HZ} and \cite[Fact~3.4]{GHX}, will be useful.

\begin{lemma}[\cite{HZ,GHX}]
\label{lem:ghx-partition}
Under the above setup, the following hold:
\begin{enumerate}[label=\textup{(\roman*)},leftmargin=2em]
    \item If $A\cap B'\ne\varnothing$, then $B\subseteq B'$.
    \item $|A\setminus A'|, |B\setminus B'|, |A'\setminus A|, |B'\setminus B|\le \eps_2 |B|$ and $|V_0|\le 2\eps_2|B|$.
    \item For every $0\le t\le \ell$, the number of $\eps_1$-bad $t$-sets contained in $B$ is at most
    $\eps_1^3\binom{|B|}{t}$.
\end{enumerate}
\end{lemma}
We shall also need the following simple shadow bound for hypergraphs with no long tight path. For a family $\mathcal F$ of $r$-sets, its shadow is defined by
\[
\partial \mathcal F
=
\left\{
S\in \binom{V(\mathcal F)}{r-1}:
S\subseteq F \text{ for some } F\in \mathcal F
\right\}.
\]
\begin{lemma}
\label{lem:shadow}
Let $r\ge2$ and let $\calF$ be an $r$-uniform hypergraph.  If $\calF$ contains no tight path with $k$ edges, then
\[
        |\calF|\le (k-1)|\partial\calF|.
\]
\end{lemma}

\begin{proof}
Suppose, for a contradiction, that $ |\mathcal F|>(k-1)|\partial\mathcal F|$.  We iteratively delete all edges containing a shadow \(S\in\partial\mathcal F\) of  degree at most \(k-1\). The process leaves a non-empty subgraph \(\mathcal G\subseteq \mathcal F\) satisfying $\deg_{\mathcal G}(S)\ge k$ for every $S\in\partial\mathcal G$.  

We now greedily construct a tight path in \(\mathcal G\). Suppose that, for some \(1\le j\le k-1\), we have constructed distinct vertices $x_1,\ldots,x_{r+j-1}$ such that $ \{x_i,\ldots,x_{i+r-1}\}\in\mathcal G$ for all $1\le i\le j$. Set $S=\{x_{j+1},\ldots,x_{j+r-1}\}$. Since \(\deg_{\mathcal G}(S)\ge k\), there exists an edge \(S\cup\{y\}\in\mathcal G\) with $y\notin\{x_1,\ldots,x_{r+j-1}\}$. Put \(x_{r+j}=y\). Repeating the argument yields distinct vertices $x_1,\ldots,x_{r+k-1}$ with $\{x_i,\ldots,x_{i+r-1}\}\in\mathcal G$ for every $1\le i\le k$, which is a  tight path of length $k$,  contradicting our initial assumption. Therefore, the inequality $ |\mathcal F|\le (k-1)|\partial\mathcal F|$ must hold. 
\end{proof}

\section{Proof of Theorem~\ref{thm:paths}}
Throughout the proof, we adopt the setup established in Section~\ref{sec2}.   Let $N=|B|$ and $q=|A\cap B'|$.
Since \(q>0\), Lemma~\ref{lem:ghx-partition} implies that  $B'=B\cup(A\cap B')$, and consequently  $|B'\setminus B|=q\le \eps_2N$.

We now construct the desired paths greedily. Suppose that fewer than \(sq\) paths have been selected so far, and let \(U\) denote the set of vertices occupied by them. Each chosen path has length at most \(s\), and therefore contains at most \(k+(s-1)(k-\ell)\le 2k\) vertices.  It follows that \(|U|<2ksq\). Our goal is to show that \(H[B'\setminus U]\) still contains an \(\ell\)-path of length at most \(s\) with two linkable ends.

Let $\mathcal{B}$ be the set of all $(k-1)$-subsets of $B'$, which contain an $\eps_1$-bad subset.
\begin{claim}\label{lem:bad-face-bound}
 $|\mathcal{B}\cap\binom B{k-1}|\le
        2^{k-1}\eps_1^3\binom N{k-1}$ and $|\mathcal{B}|
        \le
        \left(2^{k-1}\eps_1^3+2^k k\eps_2\right)
        \binom N{k-1}$.
\end{claim}

By Lemma~\ref{lem:ghx-partition}(iii), the number of \(\eps_1\)-bad \(t\)-subsets \(T\subseteq B\) is at most \(\eps_1^3\binom Nt\) for each $1\le t\le\ell$.  Thus the number of $(k-1)$-subsets \(K\subseteq B\) containing at least an \(\eps_1\)-bad subset
 is bounded above by
\[
        \sum_{t=1}^{\ell}\eps_1^3\binom Nt\binom{N-t}{k-1-t}
        \le
        \eps_1^3\sum_{t=0}^{k-1}\binom{k-1}{t}\binom N{k-1}
        =
        2^{k-1}\eps_1^3\binom N{k-1}.
\]
Since $|B'\setminus B|=q\le \eps_2N$, we have
\[
     |\mathcal{B}|\le 2^{k-1}\eps_1^3\binom N{k-1}+  \eps_2N\binom{(1+\eps_2)N}{k-2}
        \le (2^{k-1}\eps_1^3+2^k k\eps_2)\binom N{k-1}.
\]
This completes the proof of Claim 3.1.
\medskip

We shall use the following observation.

\begin{claim}\label{clm:centered-path}
Let \(v\in B'\setminus U\), and let \(\calL_v\) be the $(k-1)$-graph with vertex set $B'\setminus (U\cup\{v\})$ and edge set $N_H(v)\cap ({B'\setminus U\choose k-1}\setminus\mathcal{B})$. 
%\((k-1)\)-graph such that every\(K\in E(\calL_v)\) satisfies \(K\subseteq B'\setminus U\), \(K\cup\{v\}\in E(H)\), and \(K\notin\mathcal{B}\). 
If \(\calL_v\) contains a tight path of length $k$, then
\(H[B'\setminus U]\) contains an \(\ell\)-path of length \(s\) with two linkable ends.
\end{claim}

Let the tight path in \(\calL_v\) have ordered vertices
\(y_1,\ldots,y_{2k-2}\) and  edges $K_i=\{y_i,\ldots,y_{i+k-2}\}$, for $1\le i\le k$.
Inserting \(v\) between \(y_{k-1}\) and \(y_k\) yields an ordering \(x_1,\ldots,x_{2k-1}\) with \(x_k=v\), such that
\[
\{x_i,\ldots,x_{i+k-1}\}=K_i\cup\{v\}\in E(H)
\]
for every \(1\le i\le k\). Consequently, these \(k\) edges form a tight path in \(H[B'\setminus U]\).
From this tight path, we extract the edges whose starting positions are
 $1+j(k-\ell)$ for $0\le j\le s-1$. Since \(k=(s-1)(k-\ell)+r\) with \(1\le r\le k-\ell-1\), all these starting positions are at most \(k\). Therefore the selected \(s\) edges are well-defined and form an \(\ell\)-path.

It remains to verify that both ends of this 
$\ell$-path are linkable.   The first end lies
before \(x_k=v\), while the last end starts at
$s(k-\ell)+1>k$.  Since
\(K_i\notin\mathcal{B}\), every subset of each end 
is \(\eps_1\)-good.  Hence both ends are linkable, as required. 
This completes the proof of Claim 3.2.

\begin{claim}\label{clm:one-step}
\(H[B'\setminus U]\) contains an \(\ell\)-path of length
at most \(s\) with two linkable ends.
\end{claim}

Let
$\calC_U={B\setminus U\choose k-1}-\mathcal{B}$.
  %   =        \left\{        K\in\binom{B\setminus U}{k-1}:        K\notin\mathcal{B}        \right\}$.
Since $|U|<2ksq\le 2ks\eps_2N$, by Claim~\ref{lem:bad-face-bound}
for sufficiently large \(N\), we have
\begin{equation}
        \label{eq:CU-lower}
|\calC_U|
        \ge
        \binom N{k-1}
        -
        |U|\binom N{k-2}
        -
        2^{k-1}\eps_1^3\binom N{k-1}\ge
        \left(1-4k^2s\eps_2-2^{k-1}\eps_1^3\right)
        \binom N{k-1}.
\end{equation}
For every \(K\in\calC_U\), since $B'=B\cup(A\cap B')$, we have
\begin{equation}
\label{eq:q-neighbor}
        \deg_H(K,B')
        \ge
        \delta_{k-1}(H)-|V\setminus B'|
        \ge
        |A|-(|A|-q)
        =
        q.
\end{equation}
Recall that for $x\in B'$,  \(\deg_H(x,B)\leq\eps_1\binom N{k-1}\). Since \(U\subseteq B'\), 
the number of pairs \((K,u)\) with \(K\in\calC_U\), \(u\in U\), and
\(K\cup\{u\}\in E(H)\) is at most
\begin{equation}\label{del_num}
\sum_{u\in U}\deg_H(u,B)
        \le
        |U|\eps_1\binom N{k-1}
        \le
        2ksq\eps_1\binom N{k-1}.
\end{equation}

Let \(H_U^*\) be the \(k\)-graph on vertex set \(B'\setminus U\) whose edges are precisely those \(e \in E(H[B'\setminus U])\) for which \(\binom{e}{k-1} \cap \mathcal{C}_U \neq \emptyset\).
Then by \eqref{eq:CU-lower}-\eqref{del_num}, we have
\begin{align*}
        e(H_U^*)
        \ge
        \frac1k
        \left(q|\calC_U|-2ksq\eps_1\binom N{k-1}\right) 
        \ge
        \frac qk
        \left(1-4k^2s\eps_2-2^{k-1}\eps_1^3-2ks\eps_1\right)
        \binom N{k-1}.                                  
\end{align*}
Taking $\eps_0$ sufficiently small, we have
\begin{equation}
\label{eq:Hstar-lower}
        e(H_U^*)\ge \frac q{2k}\binom N{k-1}
        \ge \frac1{2k}\binom N{k-1}.
\end{equation}

If some \(e\in E(H_U^*)\) contains two distinct \((k-1)\)-subsets
\(e\setminus\{u\}\) and \(e\setminus\{v\}\) with
\(e\setminus\{u\}\notin\mathcal{B}\) and \(e\setminus\{v\}\notin\mathcal{B}\),
then every \(\eps_1\)-bad subset of \(e\) of size at most \(\ell\) must contain
both \(u\) and \(v\).  Order the vertices of \(e\) with \(u\) first and \(v\) last.
Then \(e\) itself forms a one-edge \(\ell\)-path whose first end avoids \(v\)
and whose last end avoids \(u\).  Hence both ends are linkable, and we are
done.

Hence we may assume that every edge of \(H_U^*\) contains a unique
\((k-1)\)-subset not in $\mathcal{B}$.  For each \(v\in B'\setminus U\), let \(\mathcal{Q}_v\) be the
\((k-1)\)-graph with vertex set $B'\setminus (U\cup\{v\})$ and edge set $\mathcal{C}_U\cap N_{H_U^*}(v)$. %consisting of all \(K\in\calC_U\) such that  K\cup\{v\}\in E(H_U^*)$.
Then 
\begin{align}\label{eq:5}
e(H_U^*)=\sum_{v\in B'\setminus U}|\mathcal{Q}_v|.
\end{align}

If for some \(v\in B'\setminus U\) the auxiliary graph \(\mathcal{Q}_v\) contains a tight path of length \(k\), then Claim~\ref{clm:centered-path} immediately yields the desired \(\ell\)-path. 
We may therefore assume that no such tight path exists in any \(\mathcal{Q}_v\). Applying Lemma~\ref{lem:shadow} to each \(\mathcal{Q}_v\), we obtain
for every $v\in B'\setminus U$, we have 
\begin{align}\label{eq:6}
    |\mathcal{Q}_v|\le (k-1)|\partial\mathcal{Q}_v|.
\end{align}
For every \(R\in\partial\mathcal{Q}_v\), choose \(K\in\mathcal{Q}_v\) with \(R\subseteq K\).
Then \(K\cup\{v\}\in E(H_U^*)\).  The set \(R\cup\{v\}\) must contain an
\(\eps_1\)-bad subset of size at most \(\ell\); otherwise \(K\cup\{v\}\) would
contain two distinct \((k-1)\)-subsets not in $\mathcal{B}$.  Thus for every $v\in B'\setminus U$, and every $R\in \partial\mathcal{Q}_v$, we have \(R\cup\{v\}\in\mathcal{B}\). 
Since each $(k-1)$-set in  $\mathcal{B}$  has exactly 
$k-1$ distinct $(k-2)$-subsets, we have
\begin{align}\label{eq:7}
\sum_{v\in B'\setminus U}|\partial\mathcal{Q}_v|\leq (k-1)|\mathcal{B}|
\end{align}
Therefore, by (\ref{eq:5}), we have
\begin{align*}
        e(H_U^*)&   =
        \sum_{v\in B'\setminus U}|\mathcal{Q}_v| \\
      & \le
        (k-1)\sum_{v\in B'\setminus U}|\partial\mathcal{Q}_v|\quad\mbox{(by (\ref{eq:6}))} \\
    &\le (k-1)^2 |\mathcal{B}|\quad\mbox{(by (\ref{eq:7}))}\\
       & \le
        (k-1)^2
        \left(2^{k-1}\eps_1^3+2^k k\eps_2\right)
        \binom N{k-1} \quad\mbox{ (by Claim~\ref{lem:bad-face-bound})}.
\end{align*}
This contradicts the lower bound in \eqref{eq:Hstar-lower}. Hence the desired \(\ell\)-path exists. This completes the proof of Claim 3.3. 

Claim~\ref{clm:one-step} establishes the required greedy step: whenever fewer than \(sq\) paths have been selected, we can always find an additional \(\ell\)-path in \(B'\setminus U\) of length at most \(s\), whose two ends are linkable. Iterating this procedure \(sq\) times yields \(sq\) vertex-disjoint \(\ell\)-paths in \(B'\), each of length at most \(s\) and with two linkable ends. This completes the proof of Theorem~\ref{thm:paths}. \qed

\section{Proof of Theorem~\ref{thm:main}}
\label{sec:proof-main}

Choose \(0<\Delta\ll 1\) sufficiently small so that both Lemma~\ref{lem:ghx-completion} and Theorem~\ref{thm:paths} are applicable. For this fixed \(\Delta\), let \(\gamma>0\) be the constant provided by Lemma~\ref{lem:ghx-nonext}. Finally, select \(n_0\) large enough that Lemmas~\ref{lem:ghx-nonext},~\ref{lem:ghx-completion} and Theorem~\ref{thm:paths} all hold for every \(n\ge n_0\).
 
Now let \(n\ge n_0\), and let \(H\) be an \(n\)-vertex \(k\)-graph with \(n\equiv 0\pmod{k-\ell}\) and \(\delta_{k-1}(H)\ge \frac{n}{s(k-\ell)}\). If \(H\) is not \(\Delta\)-extremal, then Lemma~\ref{lem:ghx-nonext} immediately yields a Hamilton \(\ell\)-cycle in \(H\).

We may therefore assume that \(H\) is \(\Delta\)-extremal. Let \(V(H)=A\cup B\) be the extremal partition chosen as in Section~\ref{sec2}, and define \(A'\), \(B'\), and \(V_0\) accordingly. Put \(q=|A\cap B'|\). If \(q=0\), then the hypothesis of Lemma~\ref{lem:ghx-completion} is vacuously satisfied (with zero paths required), and hence \(H\) already contains a Hamilton \(\ell\)-cycle.

It remains to handle the case \(q>0\). By Theorem~\ref{thm:paths}, there exist \(sq\) vertex-disjoint \(\ell\)-paths in \(B'\), each of length at most \(s\) and with two linkable ends. Applying Lemma~\ref{lem:ghx-completion} then yields a Hamilton \(\ell\)-cycle in \(H\) as well.
This completes the proof of Theorem~\ref{thm:main}.
\qed
\bibliographystyle{plain}
\bibliography{ref}

\end{document}